\documentclass[11pt,a4paper,leqno]{amsart}
\usepackage{common}

\title{Note on the compactness of commutators}

\author[Tuomas Oikari]{Tuomas Oikari}
\address[T.O.]{Department of mathematics and statistics, P.O.B. 68 (Pietari Kalmin katu 5), FI-00014 University of Helsinki, Finland}
\email{tuomas.v.oikari@helsinki.fi}

\thanks{T.O. thanks Emiel Lorist for discussions on compact operators and the anonymous referee for valuable comments that improved the clarity and presentation of the article. T.O. was supported by the Finnish Academy of Sciences and Letters and the Research Council of Finland (project no. 358180).}

\date{\today}


\begin{document}

	\maketitle

\begin{abstract}
	 The optimal sufficient conditions for the $L^p$-to-$L^q$ compactness of commutators of singular integral operators of both Calderón-Zygmund and of rough type are shown
	  in the different exponent ranges $``q>p"$, $``q=p"$ and $``q<p"$ to quickly follow from each other. 
	  The approach is through classical compactness interpolation methods. We also present a new elementary ``off-diagonal to diagonal'' extrapolation principle for the compactness of commutators of linear operators, which is of independent interest.
\end{abstract}

\section{Introduction}

\allowdisplaybreaks

The boundedness of commutators $[b,T]f := bTf-T(bf)$ of singular integral operators (SIOs) $T$ and functions $b:\R^n\to\C$ (the symbol of the commutator) has been an active research area since the classical work of Nehari \cite{Nehari1957}. By now it is well understood that the symbol heavily influences the operator theoretic properties of the commutator. The off-diagonal characterization of boundedness recently culminated in the following result. For $p,q\in (1,\infty),$ which we assume throughout the article, and a wide class of bounded SIOs $T,$ there holds that 
\begin{align}\label{eq:comm:bdd}
	\|[b,T]\|_{L^p(\R^n)\to L^q(\R^n)} \lesssim	\| b\|_{X^{p,q}(\R^n)},\quad 	X^{p,q}(\R^n) = \begin{cases}
		\BMO^{\alpha(p,q)}(\R^n) & q \geq  p, \\
		\dot{L}^{r(p,q)}(\R^n),  & q<p,
	\end{cases}
\end{align}
where the exponents are defined through the relations 
\begin{align*}
	\frac{\alpha(p,q)}{n} = \frac{1}{p}-\frac{1}{q},\qquad \frac{1}{q}= \frac{1}{r(p,q)}+ \frac{1}{p},
\end{align*}
and the function spaces  in the scale $X^{p,q}(\mathbb{R}^n)$ through the seminorms
\begin{align}\label{eq:BMOalpha}
	\|b\|_{\op{BMO}^{\alpha}(\mathbb{R}^n)} = \sup_{Q} \ell(Q)^{-\alpha}\fint_Q|b-\ave{b}_Q|,\qquad 	\|b\|_{\dot{L}^r(\mathbb{R}^n)} = \inf_{c}\| f-c\|_{L^r(\mathbb{R}^n)}.
\end{align}
Above, the supremum is taken over all cubes, $\ell(\cdot)$ stands for side length and $\fint_{(\cdot)}b, \ave{b}_{(\cdot)}$ both denote the integral average. 
When $\alpha = 0,$ the classical space of bounded mean oscillations $\op{BMO}(\mathbb{R}^n)$ is recovered, while for $\alpha>0,$ there holds that $\BMO^{\alpha}(\mathbb{R}^n) = \dot{C}^{0,\alpha}(\R^n)$ is the H\"older space, see e.g. Meyers \cite{Mey1964}. 
Emphasizing the recognition of the correct conditions on the symbol, the result \eqref{eq:comm:bdd} is due to Nehari \cite{Nehari1957} ($n=1$) and Coifman et al. \cite{CRW} ($n\geq 2$), when $q=p;$ due to Janson \cite{Jan1978}, when $q>p;$ and due to Hyt\"{o}nen \cite{HyLpq2021}, when $q<p.$
The bounds in \eqref{eq:comm:bdd} are optimal, meaning that the upper bound $(\lesssim)$ is also a lower bound $(\gtrsim)$, if the kernel of the multiscale operator $T$ is non-degenerate on all scales and in all locations, see Hyt\"onen \cite{HyLpq2021} for a formulation of a sufficient amount of non-degeneracy.

If $\alpha(p,q) > 1,$ then $b\in \dot C^{0,\alpha(p,q)}(\R^n)$ results in zero Lipschitz constant, hence the symbol is constant and the commutator is the zero operator. Therefore only the range $\alpha(p,q)\leq 1$ is interesting to study.
When $\alpha(p,q) = 1,$ it turned out in Guo et al. \cite[Theorems 1.7. \& 1.8.]{GHWY21} that the commutator is compact only if the symbol is constant, again leading to the zero operator. Compactness in the remaining range $\alpha(p,q)<1,$ which we begin to assume now, can be formulated as follows. There holds that 
\begin{align}\label{eq:comm:comp}
	[b,T]\in\calK(L^{p}(\R^n), L^q(\R^n)) \Longleftarrow b\in Y^{p,q}(\R^n) := \overline{C^{\infty}_c(\R^n)}^{X^{p,q}(\R^n)},
\end{align}
for the same classes of SIOs $T$ for which \eqref{eq:comm:bdd} is known to hold.
Above $\calK$ stands for the class of compact operators between the given spaces. 
The implication \eqref{eq:comm:comp} is due to Uchiyama \cite{Uch1978}, when $p=q$; due to Guo, He, Wu and Yang \cite{GHWY21}, when $q>p;$ and due to Hyt\"{o}nen, Li, Tao and Yang \cite{HLTY2023}, when $q<p.$ 
Moreover, the result \eqref{eq:comm:comp} is optimal, meaning that the reverse implication $(\Longrightarrow)$ holds, provided the kernel of $T$ is appropriately large, again see \cite{HyLpq2021}.

The line \eqref{eq:comm:comp} states the commutators compactness under nice approximability of the symbol in the corresponding class for boundedness. This can be further characterized as
\begin{align}\label{eq:Ypq=Testing}
 	Y^{p,q}(\R^n) = \begin{cases}
 		\VMO^{\alpha(p,q)}(\R^n) & q \geq  p, \\
 		\dot{L}^{r(p,q)}(\R^n),  & q<p.
 	\end{cases}
\end{align}
The left-hand side of \eqref{eq:Ypq=Testing} is better for proving that $b\in Y^{p,q}(\mathbb{R}^n)$  is sufficient for $[b,T]\in\mathcal{K}(L^p(\mathbb{R}^n),L^q(\mathbb{R}^n))$, whereas the right-hand side is better for the necessity direction. 
When $q=p$ the result \eqref{eq:Ypq=Testing} is in \cite{Uch1978};  when $q>p$ in \cite{GHWY21}, see also \cite{MudOik24} for an alternative approach; and when $q<p$ it is the density of $C^{\infty}_c(\mathbb{R}^n)$ in $L^r(\mathbb{R}^n),$ for $r\in [1,\infty).$ 
For the readers convenience we recall that $b\in \op{VMO}^{\alpha}(\mathbb{R}^n)$ if
\begin{align*}
   \lim_{|Q| \to \{0,\infty\} } 	\ell(Q)^{-\alpha}\fint_Q|b-\ave{b}_Q| + \lim_{ \dist(Q,0) \to\infty} 	\ell(Q)^{-\alpha}\fint_Q|b-\ave{b}_Q| = 0.
\end{align*}

The shared dense subspace $C^{\infty}_c(\mathbb{R}^n)$ of all $Y^{p,q}(\mathbb{R}^n)$ (for $\alpha(p,q)<1$) hints of a connection between the sufficiency characterizations in the different exponent ranges on the line \eqref{eq:comm:comp}. 
The working machine behind is the following classical Theorem \ref{thm:InterComp} due to Krasnosel'skii \cite{Kras60} on the interpolation of compact operators. 
\begin{thm}[\cite{Kras60}]\label{thm:InterComp} Let $p_i,q_i \in (1,\infty)$ and $T$ be a linear operator that is $L^{p_0}(\mathbb{R}^n)\to L^{q_0}(\mathbb{R}^n)$ compact and $L^{p_1}(\mathbb{R}^n)\to L^{q_1}(\mathbb{R}^n)$ bounded. Let $\theta\in (0,1)$ be arbitrary and set $1/p := \theta/p_1 + (1-\theta)/p_2$ and $1/q := \theta/q_0 + (1-\theta)/q_1.$ Then, $T$ is $L^{p}(\mathbb{R}^n)\to L^{q}(\mathbb{R}^n)$ compact.
\end{thm}
Using the identity \eqref{eq:Ypq=Testing}, stating the existence of a common shared dense subspace (i.e. of $C^{\infty}_c(\mathbb{R}^n)$), coupled with Theorem \ref{thm:InterComp},
we show in this article that the result in Guo et al. \cite{GHWY21} ($``q>p"$), in Uchiyama \cite{Uch1978} ($``q=p"$) and in Hyt\"onen et al. \cite{HLTY2023} ($``q<p"$) all follow rather immediately from each other, the simple proofs are contained in Section \ref{sect:apply}.

\subsection{An elementary compactness extrapolation principle for commutators}
Theorem \ref{thm:InterComp} can be seen to be somewhat heavy and it is interesting to ask if there exists a more elementary proof of the aforementioned implications.
Towards this goal we provide a simple ``off-diagonal to diagonal'' compactness extrapolation principle for commutators of linear operators, see Lemma \ref{lem:extranice} below, that allows us to deduce compactness in the range $``q\geq p"$ from the range $``q\leq p",$ giving, in particular, an elementary proof that the result in Hyt\"onen et al. \cite{HLTY2023} ($``q<p"$) follows from Uchiyama \cite{Uch1978} ($``q=p"$), which further follows from Guo et al. \cite{GHWY21} ($``q>p"$).
Moreover, after the discovery of Lemma \ref{lem:extranice} in the present article it was applied in the context of bi-commutators in \cite[Theorem 4.6.]{MaOi24}. 

	Since the exact choice of the function spaces and the linear operator involved is not important, we provide a small generalization by fixing the underlying data as follows.
	Consider $X^{s,t},$ a given scale of function spaces,  $Z \subset \cap X^{s,t}$ a common subset with the property that if $f\in Z,$ then $|\supp(f)|< \infty,$ and let $Y^{s,t}$ stand for the closure of $Z$ with respect to $X^{s,t}.$ 
\begin{lem}[Off-diagonal to diagonal compactness extrapolation principle for commutators]\label{lem:extranice}  
	Fix $p,q\in (0,\infty).$ Let $T$ be a linear operator so  that
	\begin{align}\label{eq:extrap1}
			\|[f,T]\|_{L^p\to L^q} \lesssim \| f\|_{X^{p,q}},  \text{ for all } f\in X^{p,q},
	\end{align}
	and suppose that there exist exponents $e\in [q,\infty]$ and $u\in (0,p]$ so that
	 \begin{align}\label{eq:extrap2}
	 	[f, T] \in \mathcal{K}(L^p, L^e) \cap \mathcal{K}(L^u,L^q), \text{ for all } f\in Z.
	 \end{align}
 Then, there holds that $[b,T]\in \mathcal{K}(L^p, L^q),$ provided that $b\in Y^{p,q}.$
\end{lem}
\begin{proof}
	As  $b\in Y^{p,q}\subset X^{p,q},$ for any $\varepsilon>0,$ there exists $f\in Z \subset X^{p,q}$ so that  by \eqref{eq:extrap1},
	\begin{align*}
		\Norm{[b,T]- [f,T]}{L^p\to L^q} =\Norm{[b-f,T]}{L^p\to L^q} \lesssim \Norm{b-f}{X^{p,q}} < \varepsilon.
	\end{align*}
	As compact operators form a closed subspace of bounded operators, it is enough to show that $[f,T]\in \mathcal{K}(L^p, L^q).$ 
	Let  $F := \supp(f)$ and by linearity of $T$ write
	\begin{align*}
		\begin{split}
			T = 1_FT1_F+1_FT1_{F^c}+1_{F^c}T1_F+1_{F^c}T1_{F^c}.
		\end{split}
	\end{align*}
	Provided each of the four terms commuted with $f$ is compact, the proof would be concluded.
	First, there holds that $[f,1_{F^c}T1_{F^c}] =  f1_{F^c}T1_{F^c}-1_{F^c}Tf1_{F^c}= 0,$ which is compact.  For the two terms in the middle, factor their commutators out as 
	\begin{align*}
		L^p\overset{1_F}{\longrightarrow} L^{u} \overset{[f,T]}{\longrightarrow} L^{q} \overset{1_{F^c}}{\longrightarrow} L^q \quad\mbox{and}\quad L^p\overset{1_{F^c}}{\longrightarrow} L^p \overset{[f,T]}{\longrightarrow} L^e \overset{1_F}{\longrightarrow} L^q.
	\end{align*}
	The multiplier $1_F$ is $L^p$-to-$L^u$ and $L^{e}$-to-$L^q$ bounded by H\"older's inequality since $q\leq e,$ $u\leq p$ and $|F|<\infty.$ The multiplier $1_{F^c}$ is $L^q$-to-$L^q$ and $L^p$-to-$L^p$ bounded trivially, as a bounded multiplier. The commutators in the diagrams are compact by \eqref{eq:extrap2}, since $f\in Z.$ As compactness is preserved by composing with bounded maps, both diagrams describe a compact map. Lastly $[f,1_FT1_F] \in  \calK(L^p,L^q),$ by either of the above diagrams. 
\end{proof}

\section{Applications}\label{sect:apply}
We let $X^{p,q}(\R^n)$ and $Y^{p,q}(\R^n)$ be as on the lines \eqref{eq:comm:bdd}, \eqref{eq:comm:comp} and $Z = C^{\infty}_c(\R^n),$ cf. formulation of Lemma \ref{lem:extranice}. Let us recall that a linear operator $T$ is a Calder\' on-Zygmund operator if it is bounded on $L^2(\R^n)$ and has the off-support representation
\begin{align*}
	Tf(x) = \int_{\R^n} K(x,y)f(y)\ud y,\qquad x\not\in\supp(f),
\end{align*}
with a kernel $K$ satisfying the size estimate $|K(x,y)|\lesssim |x-y|^{-n},$ and whenever $|x-x'|\leq 1/2|x-y|$ the smoothness estimate
\begin{align*}
	|K(x,y) - K(x',y)| +  |K(y,x) - K(y,x')| \lesssim \omega\big(|x-x'|/|x-y|\big) |x-y|^{-n}.
\end{align*}
The modulus of continuity is assumed to satisfy the Dini condition $\int_0^1\omega(t)/t\ud t < \infty.$
We denote this class of SIOs by $\op{CZ}^{\omega}(\R^n).$ 
The other class we consider are rough SIOs
\begin{align*}
	T_{\Omega}f(x) = \lim_{\varepsilon\to 0} \int_{|x-y|>\varepsilon} \frac{\Omega(x-y)}{|x-y|^n}f(y)\ud y,
\end{align*}
where the symbol $\Omega$ is assumed to be integrable and have zero mean on $\mathbb{S}^{n-1}$ and is extended to act on $\mathbb{R}^n$ by $\Omega(\lambda \cdot ) = \Omega(\cdot),$ for all $\lambda >0.$
The class $\op{RSIO}^s(\R^n)$ consists of those $T_{\Omega}$ with $\Omega\in L^s(\mathbb{S}^{n-1}).$ Provided $\Omega\in L^s(\mathbb{S}^{n-1}),$ for some $s>1,$ then $T_{\Omega}$ is $L^p$-to-$L^p$ bounded, for all $p\in (1,\infty),$ see e.g. the book of Duoandikoetxea \cite{Duoandikoetxea2001}.

Now we may precisely formulate the three distinct results on the line \eqref{eq:comm:comp} as the following theorems.
\begin{thm}[Guo et al. \cite{GHWY21}; Hyt\"onen et al. \cite{HOS2023}]\label{thm:GuoHos} Let $p,q\in (1,\infty)$ and $\alpha(p,q)\in (0,1).$ Let $T\in \op{CZ}^{\omega}(\R^n)\cup \op{RSIO}^s(\R^n),$ for some $s>p'.$ If $b\in \VMO^{\alpha(p,q)}(\R^n),$ then $[b,T] \in \calK(L^p,L^q).$
\end{thm}
\begin{thm}[Uchiyama \cite{Uch1978}]\label{thm:Uchiyama} Let $T\in \op{CZ}^{\omega}(\R^n)\cup \op{RSIO}^s(\R^n),$ for some $s>1,$ and let $p\in(1,\infty).$ If $b\in \op{VMO}(\R^n),$ then $[b,T]\in\calK(L^p(\R^n),L^p(\R^n)).$
\end{thm}
\begin{thm}[Hyt\"onen et al. \cite{HLTY2023}]\label{thm:HLTY} Let $T\in \op{CZ}^{\omega}(\R^n)\cup \op{RSIO}^s(\R^n),$ for some $s>1,$ and let $p,q\in (1,\infty)$ be such that $q<p.$ If $b\in \dot{L}^{r(p,q)}(\R^n),$ then $[b,T]\in\calK(L^p(\R^n),L^q(\R^n)).$
\end{thm}

We begin by verifying the following chain of implications
\begin{align}\label{eq:Chain1}
	\underset{``q>p"}{	\mbox{(Thm. \ref{thm:GuoHos})}} \overset{\mbox{Lem. \ref{lem:extranice}}}{\Longrightarrow }  \underset{``q=p"}{	\mbox{(Thm. \ref{thm:Uchiyama})}}\overset{\mbox{Lem. \ref{lem:extranice}}}{\Longrightarrow } \underset{``q<p"}{	\mbox{(Thm. \ref{thm:HLTY})}}.
\end{align}

\begin{proof}[Proof of Theorem \ref{thm:Uchiyama} assuming Theorem \ref{thm:GuoHos}] 
		When $T= T_{\Omega}\in  \op{RSIO}^s(\R^n)$ we take $\Omega'\in L^{\infty}(\mathbb{S}^{n-1})$ with zero mean so that $\| [b, T_{\Omega}-T_{\Omega'}]\|_{L^p\to L^q}\lesssim \| \Omega- \Omega' \|_{L^s(\mathbb{S}^{n-1})} \leq  \varepsilon.$ Thus, after approximation, we can assume that $T= T_{\Omega}\in  \op{RSIO}^{\infty}(\R^n)$
	
	Now we verify that everything falls into the framework of Lemma \ref{lem:extranice}. Let $Z = C^{\infty}_c(\mathbb{R}^n)$ and $X^{p,p} = X^{p,q} := \BMO(\R^n),$ since $p=q.$ Then, Uchiyama \cite[Lemma 3.]{Uch1978} states that $Y^{p,p}(\mathbb{R}^n) = \VMO(\R^n).$
	That \eqref{eq:extrap1} holds for rough SIOs with bounded symbols is contained in \cite{CRW}, while for CZOs with Dini kernels in Lerner et al. \cite{LOR1}.
	It remains to check that there exists $e\in [q,\infty]$ and $u\in (0,p]$ so that  \eqref{eq:extrap2} holds.
	If $s,t\in (1,\infty)$ are such that $\alpha(s,t)\in (0,1),$ then $Z\subset \op{VMO}^{\alpha(s,t)}$ (trivially). Thus the inclusion $[f, T] \in \mathcal{K}(L^s,L^t)$ for $f\in Z$ follows from Theorem \ref{thm:GuoHos}. The proof is concluded by choosing $1<u<p=q<e<\infty$ so that $\alpha(u,p), \alpha(p,e) \in (0,1),$ which is clearly possible. 
\end{proof}

\begin{proof}[Proof of Theorem \ref{thm:HLTY} assuming Theorem \ref{thm:Uchiyama}] Denote $r := r(p,q)$ and let $Z := C^{\infty}_c(\R^n)$ and $X^{p,q} := \dot L^r(\R^n).$ Since $Z$ is dense in $\dot L^r(\mathbb{R}^n),$ in fact $b\in Y^{p,q}.$ By the $L^v$-to-$L^v$ boundedness of $T,$ for all $v\in(1,\infty),$ and H\"older's inequality, the bound \eqref{eq:extrap1} is immediate.
	It remains to make the choice $e := p \in (q,\infty)$ and $u:= q \in (0,p)$ and set $Y^{p,e} = Y^{u,q} = \op{VMO}(\mathbb{R}^n).$ Since $Z\subset  \op{VMO}(\mathbb{R}^n)$ (trivially), the inclusion \eqref{eq:extrap2} follows by Theorem \ref{thm:Uchiyama}. 
\end{proof}
Next using Theorem \ref{thm:InterComp} we show that the three Theorems \ref{thm:GuoHos}, \ref{thm:Uchiyama} and \ref{thm:HLTY} all follow directly from each other (in particular reversing the chain of implications \eqref{eq:Chain1}). There are six implications to check and due to the structure of the proofs to follow we split the verification of these implications according to the following displays
\begin{align}\label{eq:Chain2}
	\underset{``q>p"}{	\mbox{(Thm. \ref{thm:GuoHos})}} \mbox{ and } 	\underset{``q=p"}{	\mbox{(Thm. \ref{thm:Uchiyama})}} \overset{\mbox{Thm. \ref{thm:InterComp}}}{\Longleftarrow } \underset{``q<p"}{	\mbox{(Thm. \ref{thm:HLTY})}} \mbox{ or } 	\underset{``q=p"}{	\mbox{(Thm. \ref{thm:Uchiyama})}}
\end{align}
and
\begin{align}\label{eq:Chain3}
	\underset{``q>p"}{	\mbox{(Thm. \ref{thm:GuoHos})}} \mbox{ or } 	\underset{``q=p"}{	\mbox{(Thm. \ref{thm:Uchiyama})}} \overset{\mbox{Thm. \ref{thm:InterComp}}}{\Longrightarrow } \underset{``q<p"}{	\mbox{(Thm. \ref{thm:HLTY})}} \mbox{ and } 	\underset{``q=p"}{	\mbox{(Thm. \ref{thm:Uchiyama})}}.
\end{align}

\begin{proof}[Proof of \eqref{eq:Chain2}, i.e. of Theorems \ref{thm:GuoHos} and \ref{thm:Uchiyama} assuming Theorem \ref{thm:HLTY} or  \ref{thm:Uchiyama}]
Fix $p,q\in (1,\infty)$ such that $\alpha(p,q)\in [0,1),$ i.e. $q\geq p,$ in particular.  Let $b\in \op{\VMO}^{\alpha(p,q)} = Y^{p,q}(\mathbb{R}^n),$ where the identity is due to Uchiyama \cite{Uch1978}, when $q= p,$ and due to Guo et al. \cite{GHWY21}, when $q>p.$ By $b\in Y^{p,q}(\mathbb{R}^n)$ and \eqref{eq:comm:bdd} find $f\in C^{\infty}_c(\mathbb{R}^n)$ so that 
\begin{align}
	\|[b-f, T]\|_{L^p\to L^q}\lesssim \|b-f\|_{\operatorname{BMO}^{\alpha(p,q)}} \leq \varepsilon.
\end{align}
Thus it is enough to show that $[f,T]\in \mathcal{K}(L^p,L^q),$ which we do by Theorem \ref{thm:InterComp}. 

Let us first assume that Theorem \ref{thm:HLTY} holds.
Let $q_0 \in (1,p)$ be arbitrary. Let $q_1\in (q,\infty )$ be such that $\alpha(p,q_1)\in (0,1),$  which is clearly possible by $\alpha(p,q)\in [0,1).$ Since $C^{\infty}_c \subset \dot{L}^{r(p,q_0)}$ (since $1<q_0< p$) it follows from Theorem \ref{thm:HLTY} that $[f,T]\in \mathcal{K}(L^{p},L^{q_0}).$ On the other hand since $C^{\infty}_c\subset \op{BMO}^{\alpha(p,q_1)}$ it follows that $[f,T]$ is $L^p\to L^{q_1}$ bounded. Since $q\in (q_0,q_1)$ (by $q_0<p\leq q < q_1$) clearly there exists $\theta\in (0,1)$ so that $1/q = \theta/q_0+(1-\theta)/q_1,$ thus  $[f,T]\in \mathcal{K}(L^p,L^q),$ by Theorem \ref{thm:InterComp}.

Then, assume that Theorem  \ref{thm:Uchiyama} holds.
Let $q_0 = p.$
Let $q_1\in (q,\infty )$ be such that $\alpha(p,q_1)\in (0,1),$  which is clearly possible by $\alpha(p,q)\in [0,1).$ Since $C^{\infty}_c \subset \operatorname{VMO},$ it follows from Theorem \ref{thm:Uchiyama} that $[f,T]\in \mathcal{K}(L^p,L^{q_0})$ (since $q_0 =p$). On the other hand since $C^{\infty}_c\subset \op{BMO}^{\alpha(p,q_1)}$ it follows that $[f,T]$ is $L^p\to L^{q_1}$ bounded. Since $q\in (q_0,q_1)$ (by $q_0<p\leq q < q_1$) clearly there exists $\theta\in (0,1)$ so that $1/q = \theta/q_0+(1-\theta)/q_1,$   thus  $[f,T]\in \mathcal{K}(L^p,L^q),$ by Theorem \ref{thm:InterComp}.
\end{proof}

\begin{proof}[Proof of \eqref{eq:Chain3}, i.e. of Theorems \ref{thm:Uchiyama} and \ref{thm:HLTY} assuming Theorem \ref{thm:Uchiyama} or  \ref{thm:GuoHos}]
	Fix $p,q\in (1,\infty)$ such that $\alpha(p,q)\leq 0,$ i.e. $q\leq p.$ If $q=p,$ let $b\in \op{\VMO} = Y^{p,q}(\mathbb{R}^n),$ where the identity is due to Uchiyama \cite{Uch1978}; and if $q<p,$ let $b\in \dot{L}^{r(p,q)} = Y^{p,q}(\mathbb{R}^n),$ where the identity is the density of $C^{\infty}_c(\mathbb{R}^n)$ in $L^{s},$ for $s\in [1,\infty).$
	 By $b\in Y^{p,q}(\mathbb{R}^n)$ and by \eqref{eq:comm:bdd}, find $f\in C^{\infty}_c(\mathbb{R}^n)$ so that 
	\begin{align}
		\|[b-f, T]\|_{L^p\to L^q}\lesssim \|b-f\|_{X^{p,q}} \leq \varepsilon.
	\end{align}
	Thus it is enough to show that $[f,T]\in \mathcal{K}(L^p,L^q),$ which we do by Theorem \ref{thm:InterComp}.
	
	Assume first that Theorem \ref{thm:Uchiyama} holds and we show that Theorem \ref{thm:HLTY} holds, thus let $q<p.$
	Let $q_0 = p$  and $q_1\in (1,q)$ be arbitrary. Since $f\in C^{\infty}_c\subset \mathrm{VMO},$ it follows by Theorem \ref{thm:Uchiyama} that $[f,T]\in\mathcal{K}(L^{p},L^{q_0})$ (by $q_0 = p$). On the other hand, since $f\in C^{\infty}_c\subset  \dot{L}^{r(p,q_1)}$ (by $1<q_1<q<p$) it follows that $[f,T]$ is $L^{p}\to L^{q_1}$ bounded. Since $q_0 = p > q > q_1,$ there exists some $\theta\in (0,1)$ so that $1/q = \theta/q_0+(1-\theta)/q_1.$ Hence it follows by Theorem \ref{thm:InterComp} that $[f,T]\in\mathcal{K}(L^p,L^q).$

	Assume then that Theorem \ref{thm:GuoHos} holds and we show that Theorems \ref{thm:Uchiyama} and \ref{thm:HLTY} hold, thus let $q\leq p.$
	Let $q_0 > p$ be such that $\alpha(p,q_0)\in (0,1),$ and let $q_1\in (1,q)$ be arbitrary. Since $f\in C^{\infty}_c\subset \mathrm{VMO}^{\alpha(p,q_0)},$ it follows by Theorem \ref{thm:GuoHos} that $[f,T] \in\mathcal{K}(L^{p},L^{q_0}).$ On the other hand since $f\in C^{\infty}_c\subset \dot{L}^{r(p,q_1)}$ (by $q_1<q\leq p$) there holds that $[f,T]$ is $L^{p}\to L^{q_1}$ bounded. Since $q_0> p  \geq q> q_1,$ there exists some $\theta\in (0,1)$ so that $1/q = \theta/q_0+(1-\theta)/q_1.$ Hence it follows from Theorem \ref{thm:InterComp} that $[f,T]\in\mathcal{K}(L^p,L^q).$ 
\end{proof}
\bibliography{references}
\end{document}